\documentclass[a4paper,11pt]{article}
\usepackage[utf8]{inputenc}

\usepackage{amssymb}
\usepackage{latexsym}
\usepackage{amsfonts}
\usepackage{amsmath}
\usepackage{amsthm} 
\usepackage{hyperref}
\usepackage{tikz}
\usepackage{tensor}
\usepackage{mathtools}
\usepackage[left=1.2in,right=1.2in,top=1.2in,bottom=1.2in]{geometry}

\usepackage[scale=0.95]{tgheros}
\usepackage[T1]{fontenc}

\DeclareMathOperator{\Diag}{Diag}

\DeclareMathOperator{\PSL}{PSL}

\DeclareMathOperator{\AGL}{AGL}

\DeclareMathOperator{\PermAut}{PermAut}
\DeclareMathOperator{\Aut}{Aut}

\DeclareMathOperator{\diff}{diff}

\DeclareMathOperator{\alt}{A}
\DeclareMathOperator{\s}{S}
\DeclareMathOperator{\soc}{soc}
\DeclareMathOperator{\Sym}{Sym}
\DeclareMathOperator{\mg}{M}

\renewcommand{\b}{\underline}

\renewcommand{\b}{\mathbf}

\renewcommand{\leq}{\leqslant}
\renewcommand{\geq}{\geqslant}

\newcommand{\D}{\mathcal{D}}
\newcommand{\F}{\mathbb{F}}

\renewcommand{\H}{\mathcal H}

\renewcommand{\O}{\mathcal O}
\renewcommand{\P}{\mathcal P}

\newcommand{\E}{\mathcal E}

\newcommand{\Z}{\mathbb{Z}}

\newcommand{\nsub}{\trianglelefteq}
\newcommand{\QQ}{ {Q_i^\times \times Q_j^\times} }

\newcommand{\Xtp}{X_{{\b 0},i,j}}
\newcommand{\Xtpi}{X_{{\b 0},i,j}^{Q_i^\times}}
\newcommand{\Xtpj}{X_{{\b 0},i,j}^{Q_j^\times}}

\DeclareMathOperator{\wt}{wt}
\DeclareMathOperator{\supp}{supp}

\theoremstyle{plain}
\newtheorem{lemma}{Lemma}
\newtheorem{theorem}[lemma]{Theorem}
\newtheorem{corollary}[lemma]{Corollary}
\newtheorem{proposition}[lemma]{Proposition}
\theoremstyle{definition}
\newtheorem{definition}[lemma]{Definition}

\numberwithin{equation}{section}
\numberwithin{lemma}{section}

\begin{document}

\title{
 $2$-Neighbour-Transitive Codes with Small Blocks of Imprimitivity\thanks{The first author is grateful for the support of an Australian Research Training Program Scholarship and a University of Western Australia Safety-Net Top-Up Scholarship. The research forms part of Australian Research Council Project FF0776186.}
}

\author{
 Neil I. Gillespie$^1$
 \and
 Daniel R. Hawtin$^2$
 \and
 Cheryl E. Praeger$^3$
}

\date{
 \small{
  \emph{
   $^1$Heilbronn Institute for Mathematical Research,\\
   School of Mathematics, Howard House,\\ 
   University of Bristol, BS8 1SN, United Kingdom.\\
   \href{mailto:neil.gillespie@bristol.ac.uk}{neil.gillespie@bristol.ac.uk}\\
   \vspace{0.25cm}
  }
 }
 \small{
  \emph{
   $^2$School of Science \& Environment (Mathematics),\\
   Memorial University of Newfoundland, Grenfell Campus,\\ 
   Corner Brook, NL, A2H 5G5, Canada.\\
   \href{mailto:dan.hawtin@gmail.com}{dan.hawtin@gmail.com}\\
   \vspace{0.25cm}
  }
 }
 \small{
  \emph{
   $^3$Centre for the Mathematics of Symmetry and Computation,\\
   The University of Western Australia,\\ 
   Crawley, WA, 6009, Australia.\\
   \href{mailto:cheryl.praeger@uwa.edu.au}{cheryl.praeger@uwa.edu.au}\\
   \vspace{0.25cm}
  }
 }
 \today
}

\maketitle
\begin{abstract}
 A code $C$ in the Hamming graph $\varGamma=H(m,q)$ is \emph{$2$-neighbour-transitive} if $\Aut(C)$ acts transitively on each of $C=C_0$, $C_1$ and $C_2$, the first three parts of the distance partition of $V\varGamma$ with respect to $C$. Previous classifications of families of $2$-neighbour-transitive codes leave only those with an affine action on the alphabet to be investigated. Here, $2$-neighbour-transitive codes with minimum distance at least $5$ and that contain ``small'' subcodes as blocks of imprimitivity are classified. When considering codes with minimum distance at least $5$, completely transitive codes are a proper subclass of $2$-neighbour-transitive codes. Thus, as a corollary of the main result, completely transitive codes satisfying the above conditions are also classified.
\end{abstract}

\section{Introduction}

Classifying classes of codes is an important task in error correcting coding theory. The parameters of perfect codes over prime power alphabets have been classified; see \cite{tietavainen1973nonexistence} or \cite{Zinoviev73thenonexistence}. In contrast, for the classes of \emph{completely regular} and \emph{$s$-regular} codes, introduced by Delsarte \cite{delsarte1973algebraic} as a generalisation of \emph{perfect} codes, similar classification results have only been achieved for certain subclasses. Recent results include \cite{borges2000nonexistence,borgesrho1,borges2012new,Borges201468}. For a survey of results on completely regular codes see \cite{borges2017survey}. Classifying families of \emph{$2$-neighbour transitive} codes has been the subject of \cite{ef2nt,aas2nt}.

A subset $C$ of the vertex set $V\varGamma$ of the Hamming graph $\varGamma=H(m,q)$ is a called \emph{code}, the elements of $C$ are called \emph{codewords}, and the subset $C_i$ of $V\varGamma$ consisting of all vertices of $H(m,q)$ having nearest codeword at Hamming distance $i$ is called the \emph{set of $i$-neighbours} of $C$. The definition of a completely regular code $C$ involves certain combinatorial regularity conditions on the \emph{distance partition} $\{C,C_1,\ldots, C_\rho\}$ of $C$, where $\rho$ is the \emph{covering radius}. The current paper concerns the algebraic analogues, defined directly below, of the classes of completely regular and $s$-regular codes. Note that the group $\Aut(C)$ is the setwise stabiliser of $C$ in the full automorphism group of $H(m,q)$.

\begin{definition}\label{sneighbourtransdef}
 Let $C$ be a code in $H(m,q)$ with covering radius $\rho$, let $s\in\{1,\ldots,\rho\}$, and $X\leq\Aut(C)$. Then $C$ is said to be
 \begin{enumerate}
  \item \emph{$(X,s)$-neighbour-transitive} if $X$ acts transitively on each of the sets $C,C_1,\ldots, C_s$,
  \item \emph{$X$-neighbour-transitive} if $C$ is $(X,1)$-neighbour-transitive, 
  \item \emph{$X$-completely transitive} if $C$ is $(X,\rho)$-neighbour-transitive, and,
  \item \emph{$s$-neighbour-transitive}, \emph{neighbour-transitive}, or \emph{completely transitive}, respectively, if $C$ is $(\Aut(C),s)$-neighbour-transitive, $\Aut(C)$-\emph{neighbour-transitive}, or $\Aut(C)$-\emph{completely transitive}, respectively.
 \end{enumerate}
\end{definition}

A variant of the above concept of complete transitivity was introduced for linear codes by Sol{\'e} \cite{sole1987completely}, with the above definition first appearing in \cite{Giudici1999647}. Note that non-linear completely transitive codes do indeed exist; see \cite{gillespie2012nord}. Completely transitive codes form a subfamily of completely regular codes, and $s$-neighbour transitive codes are a sub-family of $s$-regular codes, for each $s$. It is hoped that studying $2$-neighbour-transitive codes will lead to a better understanding of completely transitive and completely regular codes. Indeed a classification of $2$-neighbour-transitive codes would have as a corollary a classification of completely transitive codes.

Completely-transitive codes have been studied in \cite{Borges201468,Gill2017}, for instance. Neighbour-transitive codes are investigated in  \cite{ntrcodes,gillespiediadntc,gillespieCharNT}. The class of $2$-neighbour-transitive codes is the subject of \cite{ef2nt,aas2nt}, and the present work comprises part of the first author's PhD thesis \cite{myphdthesis}. Recently, codes with $2$-transitive actions on the entries of the Hamming graph have been used to construct families of codes that achieve capacity on erasure channels \cite{Kudekar:2016:RCA:2897518.2897584}, and many $2$-neighbour-transitive codes indeed admit such an action; see Proposition~\ref{ihom}.

The study of $2$-neighbour-transitive codes has been partitioned into three subclasses, as per the following definition. For definitions and notation see Section~\ref{prelimsect}.

\begin{definition}\label{efaasaadef}
 Let $C$ be a code in $H(m,q)$, $X\leq\Aut(C)$ and $K$ be the kernel of the action of $X$ on the set of entries $M$. Then $C$ is 
 \begin{enumerate}
  \item \emph{$X$-entry-faithful} if $X$ acts faithfully on $M$, that is, $K=1$,
  \item \emph{$X$-alphabet-almost-simple} if $K\neq 1$, $X$ acts transitively on $M$, and $X_i^{Q_i}$ is a $2$-transitive almost-simple group, and,
  \item \emph{$X$-alphabet-affine} if $K\neq 1$, $X$ acts transitively on $M$, and $X_i^{Q_i}$ is a $2$-transitive affine group.
 \end{enumerate}
\end{definition}

Note that Propositions~\ref{ihom} and \ref{x12trans}, and the fact that every $2$-transitive group is either affine or almost-simple (see \cite[Section 154]{burnside1911theory}), ensure that every $2$-neighbour-transitive code satisfies precisely one of the cases given in Definition~\ref{efaasaadef}. 

Those $(X,2)$-neighbour transitive codes that are also $X$-entry-faithful and have minimum distance at least $5$ are classified in \cite{ef2nt}; while those that are $X$-alphabet-almost-simple and have minimum distance at least $3$ are classified in \cite{aas2nt}. Hence, it is assumed here that the action on the alphabet is affine and the kernel of the action on entries is non-trivial. Here, $T_W$ denotes the group of translations by elements of a subspace $W$, $K$ denotes the kernel of the action of the group $X$ on entries, and $K=X\cap B$, where $B\cong \s_q^m$ is the base group in $\Aut(\varGamma)$, the full automorphism group of the Hamming graph; see Section~\ref{prelimsect}.

\begin{definition}\label{xntextensiondef}
 Let $q=p^d$, $V=\F_p^{dm}$ and $W$ be a non-trivial $\F_p$-subspace of $V$. Identify $V$ with the vertex set of the Hamming graph $H(m,q)$. An \emph{$(X,2)$-neighbour-transitive extension} of $W$ is an $(X,2)$-neighbour-transitive code $C$ containing $\b 0$ such that $T_W\leq X$ and $K=K_W$, where $K=X\cap B$, $T_W$ is the group of translations by elements of $W$ and $K_W$ is the stabiliser of $W$ in $K$. Note that $T_W\leq X$ and $\b 0\in C$ means that $W\subseteq C$. If $C\neq W$ then the extension is said to be \emph{non-trivial}.
\end{definition}

Identify $V=\F_p^{dm}$ with the vertex set of the Hamming graph $H(m,q)$, where $q=p^d$. The main result for this chapter classifies all $(X,2)$-neighbour-transitive extensions of $W$, supposing $W$ is a $k$-dimensional $\F_p$-subspace of $V$, where $k\leq d$.

\begin{theorem}\label{onedimensionaltheorem}
 Let $V=\F_p^{dm}$ be the vertex set of the Hamming graph $H(m,p^d)$ and $C$ be an $(X,2)$-neighbour-transitive extension of $W$, where $C$ has minimum distance $\delta\geq 5$ and $W$ is an $\F_p$-subspace of $V$ with $\F_p$-dimension $k\leq d$. Then $p=2$, $d=1$, $W$ is the binary repetition code in $H(m,2)$, and one of the following holds:
 \begin{enumerate}
  \item $C=W$, with $\delta=m$;
  \item $C=\H$, where $\H$ is the Hadamard code of length $12$, as in Definition~\ref{hadamarddef}, with $\delta=6$; or,
  \item $C=\P$, where $\P$ is the punctured code of the Hadamard code of length $12$, as in Definition~\ref{hadamarddef}, with $\delta=5$.
 \end{enumerate}
\end{theorem}

A corollary of Theorem~\ref{onedimensionaltheorem} regarding completely transitive codes is stated below. This result was originally proved in \cite[Theorem~10.2]{neilphd} using somewhat different methods, with the problem first being posed in \cite[Problem~6.5.4]{michealmast}. The group $\Diag_m(G)$, where $G\leq \Sym(Q)$, is defined in Section~\ref{hamminggraphautoprelim}.

\begin{corollary}\label{comptranscorr}
 Let $C$ be an $X$-completely transitive code in $H(m,2)$ with minimum distance $\delta\geq 5$ such that $K=X\cap B=\Diag_m(S_2)$. Then $C$ is equivalent to one of the codes appearing in Theorem~\ref{onedimensionaltheorem}, each of which is indeed completely transitive.
\end{corollary}

Section~\ref{prelimsect} introduces the notation used throughout the paper and Section~\ref{extsection} proves the main results.

\section{Notation and preliminaries}\label{prelimsect}

Let the \emph{set of entries} $M$ and the \emph{alphabet} $Q$ be sets of sizes $m$ and $q$, respectively, both integers at least $2$. The vertex set $V\varGamma$ of a Hamming graph $\varGamma=H(m,q)$ consists of all functions from the set $M$ to the set $Q$, usually expressed as $m$-tuples. Let $Q_i\cong Q$ be the copy of the alphabet in the entry $i\in M$ so that the vertex set of $H(m,q)$ is identified with the product 
\[
 V\varGamma=\prod_{i\in M}Q_i.
\] 
An edge exists between two vertices if and only if they differ as $m$-tuples in exactly one entry. Note that $S^\times$ will denote the set $S\setminus \{0\}$ for any set $S$ containing $0$. In particular, $Q$ will usually be a vector-space here, and hence contains the zero vector. A code $C$ is a subset of $V\varGamma$. If $\alpha$ is a vertex of $H(m,q)$ and $i\in M$ then $\alpha_i$ refers to the value of $\alpha$ in the $i$-th entry, that is, $\alpha_i\in Q_i$, so that $\alpha=(\alpha_1,\ldots,\alpha_m)$ when $M=\{1,\ldots,m\}$.  For more in depth background material on coding theory see \cite{cameron1991designs} or \cite{macwilliams1978theory}.

Let $\alpha,\beta$ be vertices and $C$ be a code in a Hamming graph $H(m,q)$ with $0\in Q$ a distinguished element of the alphabet. A summary of important notation regarding codes in Hamming graphs is contained in Table~\ref{hammingnotation}.
\begin{table}
 \begin{center}
 \begin{tabular}{cp{7 cm}}
  Notation & Explanation\\
  \hline
  $\b 0$ & vertex with $0$ in each entry\\
  
  $(a^k,0^{m-k})$ & vertex with $a\in Q$ first $k$ entries and $0$ otherwise\\
  
  $\diff(\alpha,\beta)=\{i\in M\mid \alpha_i\neq\beta_i\}$ & set of entries in which $\alpha$ and $\beta$ differ\\
  
  $\supp(\alpha)=\{i\in M\mid \alpha_i\neq 0\}$ & support of $\alpha$\\
  
  $\wt(\alpha)=|\supp(\alpha)|$ & weight of $\alpha$\\
  
  $d(\alpha,\beta)=|\diff(\alpha,\beta)|$ & Hamming distance\\
  
  $\varGamma_s(\alpha)=\{\beta\in V\varGamma \mid d(\alpha,\beta)=s\}$ & set of $s$-neighbours of $\alpha$\\
  
  $\delta=\min\{d(\alpha,\beta)\mid \alpha,\beta\in C,\alpha\neq\beta\}$ & minimum distance of $C$\\
  
  $d(\alpha,C)=\min\{d(\alpha,\beta) \mid \beta\in C\}$ & distance from $\alpha$ to $C$\\
  
  $\rho =\max\{d(\alpha,C)\mid\alpha\in V\varGamma\}$ & covering radius of $C$\\
  
  $C_s=\{\alpha\in V\varGamma \mid d(\alpha,C)=s\}$ & set of $s$-neighbours of $C$\\
  
  $\{C=C_0,C_1,\ldots, C_\rho\}$ & distance partition of $C$\\
    
  \hline
 \end{tabular}
 \caption{Hamming graph notation.}
 \label{hammingnotation}
 \end{center}
\end{table}

Note that if the minimum distance $\delta$ of a code $C$ satisfies $\delta\geq 2s$, then the set of $s$-neighbours $C_s$ satisfies $C_s=\cup_{\alpha\in C}\varGamma_s(\alpha)$ and if $\delta\geq 2s+1$ this is a disjoint union. This fact is crucial in many of the proofs below; it is often assumed that $\delta\geq 5$, in which case every element of $C_2$ is distance $2$ from a unique codeword.

A \emph{linear} code is a code $C$ in $H(m,q)$ with alphabet $Q=\F_q$ a finite field, so that the vertices of $H(m,q)$ from a vector space $V$, such that $C$ is an $\F_q$-subspace of $V$. Given $\alpha,\beta\in V$, the usual inner product is given by $\langle\alpha,\beta\rangle=\sum_{i\in M}\alpha_i\beta_i$. The \emph{dual} code of $C$ is $C^\perp=\{\beta\in V\mid \forall \alpha\in C,\langle\alpha,\beta\rangle=0\}$.

The \emph{Singleton bound} (see \cite[4.3.2]{delsarte1973algebraic}) is a well known bound for the size of a code $C$ in $H(m,q)$ with minimum distance $\delta$, stating that $|C|\leq q^{m-\delta+1}$. For a linear code $C$ this may be stated as $\delta^\perp-1\leq k\leq m-\delta+1$, where $k$ is the dimension of $C$, $\delta$ is the minimum distance of $C$ and $\delta^\perp$ is the minimum distance of $C^\perp$.

A vertex or an entire code from a Hamming graph $H(m,q)$ may be projected into a smaller Hamming graph $H(k,q)$. For a subset $J=\{j_1,\ldots,j_k\}\subseteq M$ the \emph{projection of $\alpha$}, with respect to $J$, is $\pi_J(\alpha)=(\alpha_{j_1},\ldots,\alpha_{j_k})$. For a code $C$ the \emph{projection of $C$}, with respect to $J$, is $\pi_J(C)=\{\pi_J(\alpha)\mid \alpha\in C\}$.

\subsection{Automorphisms of a Hamming graph}\label{hamminggraphautoprelim}

The automorphism group $\Aut(\varGamma)$ of the Hamming graph is the semi-direct product $B\rtimes L$, where $B\cong \Sym(Q)^m$ and $L\cong \Sym(M)$ (see \cite[Theorem 9.2.1]{brouwer}). Note that $B$ and $L$ are called the \emph{base group} and the \emph{top group}, respectively, of $\Aut(\varGamma)$. Since we identify $Q_i$ with $Q$, we also identify $\Sym(Q_i)$ with $\Sym(Q)$. If $h\in B$ and $i\in M$ then $h_i\in \Sym(Q_i)$ is the image of the action of $h$ in the entry $i\in M$. Let $h\in B$, $\sigma\in L$ and $\alpha\in V\varGamma$. Then $h$ and $\sigma$ act on $\alpha$ explicitly via: 
\begin{equation*}
\alpha^h =(\alpha_1^{h_1},\ldots,\alpha_m^{h_m})\quad\text{and}\quad
\alpha^\sigma=(\alpha_{1{\sigma^{-1}}},\ldots,\alpha_{m{\sigma^{-1}}}).
\end{equation*}
The automorphism group of a code $C$ in $\varGamma=H(m,q)$ is $\Aut(C)=\Aut(\varGamma)_C$, the setwise stabiliser of $C$ in $\Aut(\varGamma)$. 

A group acting on a set $\varOmega$ with an element or subset of $\varOmega$ appearing as a subscript denotes a setwise stabiliser subgroup, and if the subscript is a set in parantheses it is a point-wise stabiliser subgroup. A group with a set appearing as a superscript denotes the subgroup of the symmetric group on the set induced by the group. (For more background and notation on permutation groups see, for instance, \cite{dixon1996permutation}.) In particular, let $X$ be a subgroup of $\Aut(\varGamma)$. Then the \emph{action of $X$ on entries} is the subgroup $X^M$ of $\Sym(M)$ induced by the action of $X$ on $M$. Note that an element of the pre-image, inside $X$, of an element of $X^M$ does not necessarily fix any vertex of $H(m,q)$. The kernel of the action of $X$ on entries is denoted $K$ and is precisely the subgroup of $X$ fixing $M$ point-wise, that is, $K=X_{(M)}=X\cap B$. The subgroup of $\Sym(Q_i)$ induced on the alphabet $Q_i$ by the action of the stabiliser $X_i\leq X$ of the entry $i\in M$ is denoted $X_i^{Q_i}$. When $X^M$ is transitive on $M$, the group $X_i^{Q_i}$ is sometimes referred to as the \emph{action on the alphabet}. 

Given a group $H\leq \Sym(Q)$ an important subgroup of $\Aut(\varGamma)$ is the \emph{diagonal} group of $H$, denoted $\Diag_m(H)$, where an element of $H$ acts the same in each entry. Formally, define $g_h$ to be the element of $B$ with $(g_h)_i=h$ for all $i\in M$, and $\Diag_m(H)=\{g_h \mid h\in H\}$. 

It is worth mentioning that coding theorists often consider more restricted groups of automorphisms, such as the group $\PermAut(C)=\{\sigma\mid h\sigma\in\Aut(C), h=1\in B, \sigma\in L\}$. The elements of this group are called \emph{pure permutations} on the entries of the code. 

Two codes $C$ and $C'$ in $H(m,q)$ are said to be \textit{equivalent} if there exists some $x\in \Aut(\varGamma)$ such that $C^x=\{\alpha^x\mid\alpha\in C\}=C'$. Equivalence preserves many of the important properties in coding theory, such as minimum distance and covering radius, since $\Aut(\varGamma)$ preserves distances in $H(m,q)$.

\subsection{\texorpdfstring{$s$}{s}-Neighbour-transitive codes}\label{sntr}

This section presents preliminary results regarding $(X,s)$-neighbour-transitive codes, defined in Definition~\ref{sneighbourtransdef}. The next results give certain $2$-homogeneous and $2$-transitive actions associated with an $(X,2)$-neighbour-transitive code.

\begin{proposition}\cite[Proposition~2.5]{ef2nt}\label{ihom} 
 Let $C$ be an $(X,s)$-neighbour-transitive code in $H(m,q)$ with minimum distance $\delta$, where $\delta\geq 3$ and $s\geq 1$. Then for $\alpha\in C$ and $i\leq\min\{s,\lfloor\frac{\delta-1}{2}\rfloor\}$, the stabiliser $X_\alpha$ fixes setwise and acts transitively on $\varGamma_i(\alpha)$.  In particular, the action of $X_\alpha$ on $M$ is $i$-homogeneous.    
\end{proposition}

%

\begin{proposition}\cite[Proposition~2.7]{ef2nt}\label{x12trans} 
 Let $C$ be an $(X,1)$-neighbour-transitive code in $H(m,q)$ with minimum distance $\delta\geq 3$ and $|C|>1$.  Then $X_i^{Q_i}$ acts $2$-transitively on $Q_i$ for all $i\in M$.
\end{proposition}

The next result gives information about the order of the stabiliser of a codeword in the automorphism group of a $2$-neighbour-transitive code and is a strengthening of \cite[Lemma~2.10]{ef2nt}.

\begin{lemma}\label{xtpitrans}
 Let $C$ be an $(X,2)$-neighbour-transitive code in $H(m,q)$ with $\delta\geq 5$ and ${\b 0}\in C$, and let $i,j\in M$ be distinct. Then the following hold:
 \begin{enumerate}
  \item The stabiliser $\Xtp$ acts transitively on each of the sets $Q_i^\times$ and $Q_j^\times$.
  \item Moreover, $\Xtp$ has at most two orbits on $\QQ$, and if $\Xtp$ has two orbits on $\QQ$ then both orbits are the same size and $X_{\b 0}$ acts $2$-transitively on $M$.
  \item The order of $X_{\b 0}$, and hence $|X|$, is divisible by $\binom{m}{2} (q-1)^2$.
  \item If $|X_{\b 0}|=\binom{m}{2}$ then $q=2$.
 \end{enumerate}
\end{lemma}

\begin{proof}
 Now $X_{\b 0}$ acts transitively on $\varGamma_2(\b 0)$, by Proposition~\ref{ihom}, since $\delta\geq 5$. Since $|\varGamma_2({\b 0})|=\binom{m}{2} (q-1)^2$, parts 3 and 4 hold. Also, we have that the stabiliser $X_{{\b 0},\{i,j\}}$ of the subset $\{i,j\}\subseteq M$ is transitive on the set of weight $2$ vertices with support $\{i,j\}$. Hence $\Xtp$ has at most two orbits on $\QQ$ and if there are two they have equal size. Note that if $\Xtp$ has one orbit on $\QQ$ then $\Xtp$ acts transitively on each of $Q_i^\times$ and $Q_j^\times$. Suppose that $\Xtp$ has two orbits on $\QQ$, and hence that $\Xtp\neq X_{{\b 0},\{i,j\}}$. By Proposition~\ref{ihom}, $X_{\b 0}$ acts $2$-homogeneously on $M$. Since $\Xtp\neq X_{{\b 0},\{i,j\}}$, we have that $X_{\b 0}$ is in fact $2$-transitive on $M$, proving part 2. Let $k$ be the number of $\Xtp$-orbits on $Q_i^\times$. Since $X_{\b 0}$ is $2$-transitive on $M$, it follows that $\Xtpi$ is permutation isomorphic to $\Xtpj$ and hence $\Xtp$ has the same number of orbits on each of $Q_i^\times$ and $Q_j^\times$. Since each orbit of $\Xtp$ on $\QQ$ is contained in the Cartesian product of an orbit on $Q_i^\times$ with an orbit on $Q_j^\times$, it follows that $\Xtp$ has at least  $k^2$ orbits on $\QQ$. However, $k\geq 2$ implies $k^2\geq 4$, contradicting part 2, and hence part 1 holds.
\end{proof}

The concept of a design, introduced below, comes up frequently in coding theory. Let $\alpha\in H(m,q)$ and $0\in Q$. A vertex $\nu$ of $H(m,q)$ is said to be \emph{covered} by $\alpha$ if $\nu_i=\alpha_i$ for every $i\in M$ such that $\nu_i\neq 0$. A binary design, obtained by setting $q=2$ in the below definition, is usually defined as a collection of subsets of some ground set, satisfying equivalent conditions where the concept of covering a vertex corresponds to containment of a subset. We refer to the latter structures as \emph{combinatorial designs}.

\begin{definition}\label{desdef}
 A \emph{$q$-ary $s$-$(v,k,\lambda)$ design} in $\varGamma=H(m,q)$ is a subset $\mathcal{D}$ of vertices of $\varGamma_k(\b 0)$ (where $k\geq s$) such that each vertex $\nu \in\varGamma_s(\b 0)$ is covered by exactly $\lambda$ vertices of $\mathcal{D}$. When $q=2$, $\D$ is simply the set of characteristic vectors of a combinatorial $s$-design. The elements of $\D$ are called \emph{blocks}.
\end{definition}

The following equations can be found, for instance, in \cite{stinson2004combinatorial}. Let $\D$ be a binary $s$-$(v,k,\lambda)$ design with $|\D|=b$ blocks and let $r$ be the number of blocks incident with a point. Then $vr=bk$, $r(k-1)=\lambda(v-1)$ and 
\begin{equation}
 b=\frac{v(v-1)\cdots(v-s+1)}{k(k-1)\cdots(k-s+1)}\lambda.\label{designparams}
\end{equation}

The definition below is required in order to state the remaining two results of this section.

\begin{definition}\label{regcodedefinition}
 Let $C$ be a code in $H(m,q)$ with covering radius $\rho$, and $s$ be an integer with $0\leq s\leq \rho$. Then,
 \begin{enumerate}
  \item $C$ is \emph{$s$-regular} if, for each $i\in\{0,1,\ldots,s\}$, each $k\in\{0,1,\ldots,m\}$, and every vertex $\nu\in C_i$, the number $|\varGamma_k(\nu)\cap C|$ depends only on $i$ and $k$, and,
  \item $C$ is \emph{completely regular} if $C$ is $\rho$-regular.
 \end{enumerate}
\end{definition}

\begin{lemma}\cite[Lemma~2.16]{ef2nt}\label{design}
 Let $C$ be an $(X,s)$-neighbour transitive code in $H(m,q)$. Then $C$ is $s$-regular. Moreover, id $C$ has with minimum distance $\delta\geq 2s$ and contains $\b 0$, then for each $k\leq m$ the set of codewords of weight $k$ forms a $q$-ary $s$-$(m,k,\lambda)$ design, for some $\lambda$.
\end{lemma}

\begin{definition}\cite[Definition~4.1]{ef2nt}\label{hadamarddef}
 Let $\mathcal{P}$ be the punctured Hadamard $12$ code, obtained as follows (see \cite[Part 1, Section 2.3]{macwilliams1978theory}). First, we construct a normalised Hadamard matrix $H_{12}$ of order $12$ using the Paley construction. 
 \begin{enumerate}
  \item Let $M=\F_{11}\cup \{*\}$ and let $H_{12}$ be the $12\times 12$ matrix with first row $v$, where $v_a=-1$ if $a$ is a square in $\F_{11}$ (including $0$), and $v_a=1$ if $a$ is a non-square in $\F_{11}$ or $a=*\in M$, taking the orbit of $v$ under the additive group of $\F_{11}$ acting on $M$ to form $10$ more rows and adding a final row, the vector $((-1)^{12})$. 
  \item The Hadamard code $\H$ of length $12$ in $H(12,2)$ then consists of the vertices $\alpha$ such that there exists a row $u$ in $H_{12}$ or $-H_{12}$ satisfying $\alpha_a=0$ when $u_a=1$ and $\alpha_a=1$ when $u_a=-1$. 
  \item The punctured code $\P$ of $\H$ is obtained by deleting the coordinate $*$ from $M$. The weight $6$ codewords of $\P$ form a binary $2$-$(11,6,3)$ design, which we denote throughout by $\D$. The code $\P$ consists of the following codewords: the zero codeword, the vector $(1^{11})$, the characteristic vectors of the $2$-$(11,6,3)$ design $\D$, and the characteristic vectors of the complement of that design, which forms a $2$-$(11,5,2)$ design. (Both $\D$ and its complement are unique up to isomorphism \cite{todd1933}.)
  \item The even weight subcode $\mathcal E$ of $\mathcal P$ is the code consisting of the zero codeword and the $2$-$(11,6,3)$ design. 
 \end{enumerate}
\end{definition}

\begin{proposition}\cite[Proposition~4.3]{ef2nt}\label{equivpunchad}
 Let $C$ be a $2$-regular code in $H(11,2)$ with $\delta\geq 5$ and $|C|\geq 2$. Then one of the following holds: 
 \begin{enumerate}
 \item $\delta=11$ and $C$ is equivalent to the binary repetition code,
 \item $\delta=5$ and $C$ is equivalent to the punctured Hadamard code $\P$, or 
 \item $\delta=6$ and $C$ is equivalent to the even weight subcode $\mathcal{E}$ of $\P$.
 \end{enumerate}
\end{proposition}

\section{Extensions of the binary repetition code}\label{extsection}

In this section it will be shown that the hypotheses of Theorem~\ref{onedimensionaltheorem} imply that $W$ is the binary repetition code in $H(m,q)$. From there, all $(X,2)$-neighbour-transitive extensions of the binary repetition code are classified. First, a more general result regarding $(X,2)$-neighbour-transitive codes. Note that a \emph{system of imprimitivity} for the action of a group $G$ on a set $\varOmega$ is a non-trivial partition of $\varOmega$ preserved by $G$, and a part of the partition is called a \emph{block of imprimitivity}.

\begin{lemma}\label{blockofimpis2nt}
 Suppose $C$ is an $(X,2)$-neighbour transitive code with $\delta\geq 5$ and that $\Delta$ is a block of imprimitivity for the action of $X$ on $C$. Then $\Delta$ is an $(X_\Delta,2)$-neighbour transitive code with minimum distance $\delta_\Delta\geq 5$.
\end{lemma}

\begin{proof}
 Since $\Delta$ is a block of imprimitivity for the action of $X$ on $C$, it follows that $X_\Delta$ is transitive on $\Delta$. Since $\delta\geq 5$ and $\Delta\subseteq C$ it follows that $\delta_\Delta\geq 5$. Since $X_\Delta$ fixes $\Delta$, we have that $X_\Delta$ fixes $\Delta_1$ and $\Delta_2$. It remains to show that $X_\Delta$ is transitive on $\Delta_i$ for $i=1,2$. Let $i\in\{1,2\}$ and $\mu,\nu\in\Delta_i$. Then, since $\delta_\Delta\geq 5$, there exists $\alpha,\beta\in \Delta$ such that $\mu\in\varGamma_i(\alpha)$ and $\nu\in\varGamma_i(\beta)$. Moreover, $\mu,\nu\in C_i$ since $\delta\geq 5$. Hence, there exists $x\in X$ such that $\mu^x=\nu$ and, since $\delta\geq 5$, $\alpha^x=\beta$ and so lies in $\Delta\cap\Delta^x$. Since $\Delta$ is a block of imprimitivity, it follows that $x$ fixes $\Delta$ setwise, so that $x\in X_\Delta$. Thus $X_\Delta$ is transitive on $\Delta_i$ for $i=\{1,2\}$.
\end{proof}

\begin{corollary}\label{Wis2nt}
 Let $C$ be an $(X,2)$-neighbour-transitive extension of $W$ such that $C$ has minimum distance $\delta\geq 5$. Then $W$ is a block of imprimitivity for the action of $X$ on $C$ and $W$ is $(X_W,2)$-neighbour-transitive with minimum distance $\delta_W\geq 5$.
\end{corollary}

\begin{proof}
 Now, $K=K_W$ is normal in $X$ and $T_W\leq K_W$ is transitive on $W$ from which it follows that $W$ is an orbit of $K$ on $C$ and hence, by \cite[Theorem~1.6A (i)]{dixon1996permutation}, is a block of imprimitivity for the action of $X$ on $C$. Thus, the result is implied by Lemma~\ref{blockofimpis2nt}.
\end{proof}

The next result shows that the binary repetition code is the only $2$-neighbour-transitive code which is a $k$-dimensional $\F_p$-subspace of $V=\F_p^{dm}$, identified with the vertex set of $H(m,p^d)$, such that $1\leq k\leq d$. 

\begin{lemma}\label{onedimbinaryrep}
 Let $q=p^d$ and $V=\F_p^{dm}$ be the vertex set of the Hamming graph $H(m,q)$ and let $W$ be a $k$-dimensional $\F_p$-subspace of $V$, with $1\leq k\leq d$, such that $W$ is an $(X,2)$-neighbour-transitive code with minimum distance $\delta\geq 5$. Then $q=2$ and $W$ is the binary repetition code in $H(m,2)$.
\end{lemma}

\begin{proof}
 We claim that $\delta=m$. As any $(X,2)$-neighbour transitive code is also $2$-regular, by Lemma~\ref{design}, and ${\b 0}\in W$, proving the claim implies the result, by \cite[Lemma~2.15]{ef2nt}. Suppose for a contradiction that $\delta<m$. It follows that there exists a weight $\delta$ codeword $\alpha\in W$ and distinct $i,j\in M$ such that $\alpha_i=0$ and $\alpha_j\neq 0$. Now, $X_{{\b 0},i,j}$ acts transitively on $Q_j^\times$, by Lemma~\ref{xtpitrans}, so that for all non-zero $a\in\F_p^d$ there exists some $x_a\in X_{{\b 0},i,j}$ such that $\alpha^{x_a}\in W$ with $(\alpha^{x_a})_j=a$. As $a$ ranges over all non-zero $a\in\F_p^d$ this gives $p^d-1$ distinct codewords. Since $|W|=p^k\leq p^d$, and ${\b 0}\in W$, it follows that $|W|=p^d$ and $k=d$. Note that since $\alpha_i=0$ and $x_a\in X_{{\b 0},i,j}$ this implies that every element of $W$ has $i$-th entry $0$. By Proposition~\ref{ihom}, $X_{\b 0}$ is, in particular, transitive on $M$. Hence, there exists some $y=h\sigma\in X_{\b 0}$, with $h\in B$ and $\sigma\in L$, such that $j^\sigma=i$. Thus $\alpha^y\in W$ with $(\alpha^y)_i\neq 0$. This gives a contradiction, proving the claim that $\delta=m$.
\end{proof}

Lemma~\ref{onedimbinaryrep} implies part 1 of Theorem~\ref{onedimensionaltheorem} and also that, given the hypotheses of Theorem~\ref{onedimensionaltheorem}, it can be assumed that $q=2$ and $W$ is the repetition code in $H(m,2)$.

\begin{lemma}\label{kernelistranslations}
 Let $C$ be an $(X,2)$-neighbour-transitive extension of $W$, where $W$ is the repetition code in $H(m,2)$, with $\delta\geq 5$. Then $X_{\b 0}\cong X_{\b 0}^M=X_W^M$, $K=T_W$ and $X_W=T_W\rtimes X_{\b 0}$.
\end{lemma}

\begin{proof}
 Let $W$ be the repetition code in $H(m,2)$. If $x=h\sigma\in X_{\b 0}$, with $h\in B$ and $\sigma\in L$, then $q=2$ implies $h_i=1$ for all $i\in M$. Thus $X_{\b 0}\cong X_{\b 0}^M$. By Corollary~\ref{Wis2nt}, $W$ is a block of imprimitivity for the action of $X$ on $C$, from which it follows that $X_W=T_W\rtimes X_{\b 0}$, since $T_W$ acts transitively on $W$. Thus, $X_{\b 0}\cong X_{\b 0}^M= X_W^M$ and $K=T_W$.
\end{proof}
%
%
%

\begin{lemma}\label{2transitive}
 Suppose $C$ is a non-trivial $(X,2)$-neighbour transitive extension of the repetition code $W$ in $H(m,2)$, where $C$ has minimum distance $\delta\geq 5$. Then $\delta\neq m$, $X^M$ acts $2$-transitively on $M$ and $X_W^M$ acts $2$-homogeneously on $M$. Moreover, if $X_W^M$ acts $2$-transitively on $M$ then $X_{i,j}^M$ has a normal subgroup of index $2$, where $i,j\in M$ and $i\neq j$.
\end{lemma}

\begin{proof}
 First, note that $\omega\in W$ if and only if $\omega_i=\omega_j$ for all $i,j\in M$. Since $C\neq W$ there exists a codeword $\alpha\in C\setminus W$ and distinct $i,j\in M$ such that $\alpha_i=0$ and $\alpha_j=1$, since otherwise $\alpha\in W$. Note that this implies that $\delta\neq m$. Let $J=\{i,j\}\subseteq M$ and consider the projection code $P=\pi_J(C)$. Now, $\pi_J(W)=\{(0,0),(1,1)\}\subseteq P$ and $\pi_J(\alpha)=(0,1)\in P$. Also, $\beta=\alpha+(1,\ldots,1)\in C$, since $T_W\leq X$, which implies $\pi_J(\beta)=(1,0)\in P$. Thus, $P$ is the complete code in the Hamming graph $H(2,2)$. By \cite[Corollary~2.6]{ef2nt}, $X_{\{i,j\}}$ acts transitively on $C$, from which it follows that $X_{\{i,j\}}^P$ acts transitively on $P$. Thus $|P|=4$ divides $|X_{\{i,j\}}^P|$ and hence also divides $|X|$. By Lemma~\ref{kernelistranslations}, $K=T_W$ so that $|K|=2$. Thus $2$ divides $|X/K|$. Proposition~\ref{ihom} and \cite[Exercise~2.1.11]{dixon1996permutation} then imply that $X/K=X^M$ is $2$-transitive. 
 
 By Corollary~\ref{Wis2nt}, $W$ is $(X_W,2)$-neighbour-transitive. Thus, by Proposition~\ref{ihom}, $X_W^M$ is $2$-homogeneous on $M$. Suppose $X_W^M$ is $2$-transitive on $M$. Since $X_{W,\{i,j\}}^P$ contains $K$ and interchanges $i$ and $j$, $|X_{W,\{i,j\}}^P|$ is divisible by $4$. Now, $|X_{\{i,j\}}^P|\leq 8$, since $\Aut(H(2,2))=(\s_2\times \s_2)\rtimes \s_2$. Furthermore, $|X_{\{i,j\}}^P:X_{W,\{i,j\}}^P|=2$, since $X_{\{i,j\}}^P$ acts transitively on $P$. Thus $X_{\{i,j\}}^P=(\s_2\times \s_2)\rtimes \s_2$, and so $|X_{i,j}^P|=4$. Let $H$ be the kernel of the action of $X_{i,j}$ on $P$. Since the only non-identity element of $K=T_W$ acts non-trivially on $P$, we deduce that $|K^P|=2$ and $H\cap K=1$. Hence, 
 \[
  \frac{X_{i,j}^P}{K^P} \cong \frac{X_{i,j}/H}{HK/H} \cong \frac{X_{i,j}}{HK} \cong \frac{X_{i,j}/K}{HK/K} \cong \frac{X_{i,j}^M}{H^M}.
 \]
 Therefore, $X_{i,j}^M$ has a quotient of size $2$, since $|X_{i,j}^P/K^P|=2$, and thus $H^M$ is a normal subgroup of $X_{i,j}^M$ of index $2$.
\end{proof}

The socle of a finite group is the product of all its minimal normal subgroups. If $C$ is an $(X,2)$-neighbour-transitive extension of the binary repetition code $W$ in $H(m,2)$ then the next two results show that the socles of $X^M$ and $X_W^M$ cannot be equal and that the socle of $X^M$ cannot be $\alt_m$.

\begin{lemma}\label{unequalsocles}
 Let $W$ be the repetition code in $H(m,2)$ and $C$ be a non-trivial $(X,2)$-neighbour-transitive extension of $W$ with $\delta\geq 5$. Then $\soc(X/K)\neq \soc(X_W/K)$.
\end{lemma}

\begin{proof}
 Let $H\leq X$ such that $K<H$ and $H/K=\soc(X/K)$. Note that this implies that $H\nsub X$. By Lemma~\ref{kernelistranslations}, $X_W=K\rtimes X_{\b 0}$. Suppose $H/K=\soc(X_W/K)$, and note that by Lemma~\ref{2transitive}, $X_W^M=X_W/K$ acts $2$-homogeneously on $M$ and $X^M\cong X/K$ acts $2$-transitively on $M$ with the same socle.

 By considering vertices as characteristic vectors of subsets of $M$, we may identify the set of all subsets of $M$ with the vertex set $V\cong\F_2^m$ of $H(m,2)$. By Lemma~\ref{kernelistranslations}, $K=T_W\cong\Z_2$. Consider the quotient of $H(m,2)$ by the orbits of $K$, thereby identifying each subset $J$ of $M$ with its complement ${\bar J}$. In particular, $W$ is identified with $\{\emptyset,M\}$. This gives induced actions of $X$, $X_W$ and $X_{\b 0}$ on the set:
 \[
  \O=\left\{\{J,{\bar J}\}\mid J\in C\right\}.
 \]
 Note that $\O$ is a set of partitions of $M$, and $x\in X\setminus X_W$ does not necessarily fix $\{|J|,|{\bar J}|\}$. Since the single non-trivial element of $K$ maps $J\subseteq M$ to ${\bar J}$, for each $J$, it follows that $K$ fixes every element of $\O$. Thus, $K$ is in the kernel $X_{(\O)}$ of the action of $X$ on $\O$. If $x\in X\setminus X_W$, then $\{\emptyset,M\}^x\neq \{\emptyset,M\}$, so that $X_{(\O)}\leq X_W$. By Lemma~\ref{kernelistranslations}, $X_W=K\rtimes X_{\b 0}$. It follows that $X_{(\O)}/K\nsub X_W/K$, and, since $H/K=\soc(X_W/K)$, either $X_{(\O)}/K=1$, or $H/K\nsub X_{(\O)}/K$.
 
 Suppose that $H/K\leq X_{(\O)}/K$. Note that, by assumption, $C\neq W$. As $H/K$ fixes $\O$ element-wise, $H/K$ fixes the non-trivial partition $\{J,{\bar J}\}$, for each $J\in C\setminus W$. Since $H/K=\soc(X_W/K)$ acts transitively on $M$, we have that $H/K$ acts imprimitively on $M$ and $|J|=|{\bar J}|$, so that $2$ divides $m$ and $\delta=m/2$. By \cite{Kantor1969}, a $2$-homogeneous but not $2$-transitive group has odd degree, and hence the fact that $2\mid m$ implies that $X_{\b 0}$ acts $2$-transitively on $M$. By \cite[Section 134 and Theorem~IX, p.~192]{burnside1911theory}, a $2$-transitive group with an imprimitive socle has a normal subgroup of prime power order. Thus, by \cite[Section~7.7]{dixon1996permutation}, we deduce that $X_W^M$ is affine and, since $2\mid m$, we have that $X_W^M\leq \AGL_d(2)$ and $M\cong \F_2^d$. Since $X_W^M$ and $X^M$ have the same socle, $X^M$ is also an affine $2$-transitive group. Now, if $U=\{J,{\bar J}\}$ is fixed by the group of translations of $\F_2^d$ acting on $M$, then either $J$ or ${\bar J}$ is a $(d-1)$-space of $M$. Let $i=0\in M$. Then $X_{W,i}$ acts transitively on $M\setminus \{i\}$, that is, on the set of $1$-spaces of $M$. Since each $1$-space is orthogonal to a $(d-1)$-space, it follows that $X_{W,i}$ also acts transitively on the set of $(d-1)$-spaces of $M$. This implies $|\O\setminus\{\emptyset,M\}|=2^d-1$, the number of $(d-1)$-spaces in $M$. Thus, $|C|=2^d|W|$. Now $K\leq X_W\leq X$ implies $|C|/|W|=|X|/|X_W|=|X^M|/|X_W^M|$, that is, $|X^M|=2^d|X_W^M|$. This gives a contradiction, as there is no finite transitive linear group acting on $2^d-1$ points with an index $2^d$ subgroup that remains transitive on $2^d-1$ points (see \cite[Hering's Theorem]{liebeckaffinerank3}). Thus, $X_{(\O)}/K=1$.
 
 By Lemma~\ref{2transitive}, $X^M$ acts $2$-transitively on $M$. Since $H/K=\soc(X/K)$, it follows that $H/K$ acts transitively on $M$. As $X$ acts transitively on $\O$, the stabiliser in $X/K$ of any element of $\O$ is conjugate in $X/K$ to the stabiliser $X_W/K$ of $\{\emptyset,M\}\in\O$. It follows from this that $H/K$ fixes every element of $\O$, since $H/K\nsub X/K$ and $H\leq X_W/K$. If $H/K$ fixes each element of $\O$ then $H/K\leq X_{(\O)}/K$, giving a contradiction. Thus $\soc(X/K)\neq\soc(X_W/K)$.
\end{proof}

\begin{lemma}\label{onedimsocnotam}
 Let $C$ be a non-trivial $(X,2)$-neighbour-transitive extension of $W$ with $\delta\geq 5$, where $W$ is the repetition code in $H(m,2)$. Then $\soc(X^M)\neq \alt_m$. 
\end{lemma}

\begin{proof}
 Suppose $\soc(X^M)=\alt_m$. By Lemma~\ref{2transitive}, $X_W/K\cong X_W^M$ is a $2$-homogeneous group and thus primitive, and, by Lemma~\ref{unequalsocles}, $\soc(X_W/K)\neq\soc(X/K)$. By Lemma~\ref{kernelistranslations}, $|C|=|X:X_{\b 0}|=2|X:X_W|=2|X/K:X_W/K|$. Now, \cite{bochert1889ueber} (see also \cite[Theorem~14.2]{wielandt2014finite}) gives a lower bound on the index of a primitive non-trivial subgroup $G$ of the symmetric group $\s_m$, with $G$ not containing the alternating group, of $|\s_m:G|\geq \lfloor(m+1)/2\rfloor!$. Since $X_W^M$ is primitive and $X/K\cong \alt_m$ or $\s_m$, it follows that 
 \[
  |C|=2|X/K:X_W/K|\geq t|X/K:X_W/K|=|\s_m:X_W^M|\geq \lfloor(m+1)/2\rfloor!,
 \]
 where $t=1$ or $2$. However, by the Singleton bound we have $|C|\leq 2^{m-\delta+1}\leq 2^{m-4}$. Combining these two inequalities, we have $\lfloor(m+1)/2\rfloor!\leq 2^{m-4}$, which does not hold when $m\geq 5$.
\end{proof}
%

%

\begin{table}
\begin{center}
\begin{tabular}{cccc}
 $G$ & $H$ & degree\\
 \hline
 $\Z_7\rtimes \Z_3$ & $\PSL_3(2)$ & $7$\\
 $\Z_{11}\rtimes \Z_5$ & $\PSL_2(11)$ or $\mg_{11}$ & $11$\\
 $\Z_{23}\rtimes \Z_{11}$ & $\mg_{23}$ & $23$\\
 $\PSL_2(7)$ & $\AGL_3(2)$ & $8$\\
 $\alt_7$ & $\alt_8$ & $15$\\
 $\PSL_2(11)$ & $\mg_{11}$ & $11$\\
 $\PSL_2(11)$ or $\mg_{11}$ & $\mg_{12}$ & $12$\\
 $\PSL_2(23)$ & $\mg_{24}$ & $24$\\
 \hline
\end{tabular}
\caption{Groups $G<H\leq \s_m$ where $H$ is $2$-transitive, $G$ is $2$-homogeneous, $\soc(H)\neq \alt_m$ and $\soc(G)\neq \soc(H)$; see \cite[Proposition~4.4 and Table~3]{ef2nt}.}\label{2tr2hom}
\end{center}
\end{table}

The main theorem can now be proved.

\begin{proof}[Proof of Theorem~\ref{onedimensionaltheorem}]
 Suppose $C$ is an $(X,2)$-neighbour-transitive extension of $W$ with $\delta\geq 5$, where $W$ is a $k$-dimensional $\F_p$-subspace of $V=\F_p^{dm}$ and $1\leq k\leq d$. By Lemma~\ref{onedimbinaryrep}, $W$ is the binary repetition code (not just an equivalent copy of it, since $\b 0\in W$) and thus $q=2$. If $C=W$ then $C$ is a trivial extension of $W$ and outcome 1 holds. Suppose the extension is non-trivial. Then, by Lemma~\ref{2transitive}, $\delta\neq m$, $X^M$ acts $2$-transitively on $M$ and either $X_W^M$ is $2$-transitive and $X_{i,j}^M$ has an index $2$ normal subgroup, or $X_{W}^M$ acts $2$-homogeneously, but not $2$-transitively, on $M$. Also, by Lemma~\ref{unequalsocles}, the socles of $X^M$ and $X_W^M$ are not equal, and, by Lemma~\ref{onedimsocnotam}, $\soc(X^M)\neq \alt_m$. Thus, by \cite[Proposition~4.4]{ef2nt}, the possibilities for $X^M$ and $X_W^M$ are as in Table~\ref{2tr2hom}.
 
 Now $T_W\leq X$ implies that if there exists some weight $k$ codeword in $C$, then there is also a weight $m-k$ codeword. Thus $\delta\leq m/2$ and $\delta\geq 5$ implies $m\geq 10$. In particular, $X^M\neq \PSL_3(2)$ or $\AGL_3(2)$. Suppose $X^M\cong \PSL_2(11)$ and $m=11$. Then $\delta=5$ and, by Proposition~\ref{equivpunchad}, $C$ is either the punctured Hadamard code $\P$ or the even weight subcode $\E$ of the punctured Hadamard code. The even weight subcode of the punctured Hadamard code is not invariant under $T_W$, so $C\neq \E$. Moreover, as in the proof of \cite[Proposition~4.3]{ef2nt}, the only copy of $\PSL_2(11)$ in $\Aut(\P)$ fixes ${\b 0}$, and hence $X_{\b 0}^M\cong \PSL_2(11)$. This implies that $X_W^M=\PSL_2(11)$, by Lemma~\ref{kernelistranslations}, and thus $X^M=X_W^M$, a contradiction. 
 
 Suppose $m=23$, $X^M\cong \mg_{23}$ and $X_W^M\cong \Z_{23}\rtimes \Z_{11}$. By Lemma~\ref{kernelistranslations}, $X_W=T_W\rtimes X_{\b 0}$ and $K=T_W$, so that $|X_{\b 0}|=|X_W^M|$ which gives $|C|=|X|/|X_{\b 0}|=2|X^M|/|X_W^M|$, and hence $|C|=80640$. However, this contradicts the bound of $|C|\leq 24106$ for a code of length $23$ with $\delta\leq 5$ from \cite[Table~I and Theorem~1]{Best78boundsfor}.

 Suppose $m=15$, $X^M\cong \alt_8$ and $X_W^M\cong \alt_7$. Then $X_{i,j}^M\cong \alt_6$ is simple, contradicting Lemma~\ref{2transitive}.
 
 Suppose $m=11$, $X^M\cong \mg_{11}$ and $X_W^M\cong \PSL_2(11)$. Then, by Proposition~\ref{equivpunchad}, $C$ is either the punctured Hadamard code $\P$ or the even weight subcode of $\P$. The even weight subcode of $\P$ is not invariant under $T_W$, so $C=\P$. The automorphism group of $\P$ is $X=\Aut(\P)\cong 2\times \mg_{11}$ with $X_{\b 0}\cong\PSL_2(11)$ and $K=T_W$. By \cite[Theorem~1.1]{Gillespie20131394} $\P$ is an $(X,2)$-neighbour-transitive extension of $W$, as in outcome 3.
 
 Suppose $m=12$, $X^M\cong \mg_{12}$ and $X_W^M\cong \mg_{11}$ or $\PSL_2(11)$. If $X_W^M\cong \PSL_2(11)$ then, as the index of $\PSL_2(11)$ in $\mg_{12}$ is $144$, we have $|C|=288$. However, since $\delta\geq 5$, the Singleton bound gives $|C|\leq 2^{m-\delta+1}\leq 256$. Thus $X_W^M\cong \mg_{11}$ and $|C|=24$. If weight $5$ codewords exist then, by Lemma~\ref{design} and (2.1), there are 
 \[
  b=\frac{v(v-1)\lambda}{k(k-1)}=\frac{12\cdot 11\lambda}{5\cdot 4}=\frac{3\cdot 11\lambda}{5}
 \]
 of them, for some $\lambda$ divisible by $5$. Since $\lambda\geq 5$ implies $b\geq 33> |C|=24$, it follows that $\lambda=0$. Thus $\delta\geq 6$, and as $\delta\leq m/2=6$, it follows that $\delta=6$. The Hadamard code $\H$ of length $12$ with $X=\Aut(\H)\cong 2.\mg_{12}$, $X_{\b 0}\cong \mg_{12}$ and $K=T_W$ is then the unique $(X,2)$-neighbour-transitive extension of $W$ with these parameters, by \cite[Theorem~1.1]{Gillespie20131394}, as in outcome 2.
 
 Finally, suppose $m=24$, $X^M\cong \mg_{24}$ and $X_W^M\cong \PSL_2(23)$. Then $X_{i,j}^M\cong \mg_{22}$ is simple, contradicting Lemma~\ref{2transitive}.
\end{proof}

Finally, the proof of Corollary~\ref{comptranscorr} is given below.

\begin{proof}[Proof of Corollary~\ref{comptranscorr}]
 Suppose $C$ is $X$-completely transitive with minimum distance $\delta\geq 5$ such that $K=\Diag_m(S_2)$, and assume that ${\b 0}\in C$. The fact that $\delta\geq 5$ implies that $C_2$ is non-empty and thus $C$ is $(X,2)$-neighbour-transitive. Since $K\nsub X$ and $X$ acts transitively on $C$, it follows from Lemma~\ref{blockofimpis2nt} that the orbit $\Delta={\b 0}^K$ of ${\b 0}$ under $K$ is an $(X_\Delta,2)$-neighbour-transitive code. Since $K=\Diag_m(S_2)$ we have that $|\Delta|=2$ and $\Delta$ has minimum distance $m$. Thus, since any $2$-neighbour-transitive code is $2$-regular, \cite[Lemma~2.15]{ef2nt} implies that $\Delta$ is the binary repetition code in $H(m,2)$. Hence, $q=2$, $Q\cong \Z_2$ and $C$ satisfies the hypotheses of Theorem~\ref{onedimensionaltheorem}, and so is one of the codes listed there. The binary repetition code has automorphism group $\Diag_m(S_2)\rtimes \Sym(M)$ and is seen to be completely transitive by identifying the vertices of $H(m,2)$ with the subsets of $M$. By \cite[Theorem~1.1]{Gillespie20131394}, the Hadamard code of length $12$ and its punctured code are completely transitive. This completes the proof.
\end{proof}

\end{document}